\documentclass[11pt, reqno]{article}
\usepackage{hyperref}
\usepackage[utf8]{inputenc}
\usepackage[english]{babel}
\usepackage{amsmath}
\usepackage{amsfonts}
\usepackage{amscd}
\usepackage{amssymb}
\usepackage{amsthm}
\usepackage{natbib}
\usepackage[mathscr]{euscript}
\usepackage{textcomp}
\newtheorem{theorem}{Theorem}[section]
\newtheorem{lemma}{Lemma}[section]
\newtheorem{remark}{Remark}[section]

\newtheorem{corollary}{Corollary}
\newtheorem{example}{Example}
\usepackage{latexsym}

\usepackage{titlesec}
\usepackage{leftidx}
\titleformat{\section}
{\normalfont\fontsize{11}{10}\bfseries}{\thesection}{1em}{}
\usepackage{authblk}
\usepackage{enumitem}
\usepackage{xpatch}
\xpatchcmd{\author}{\relax#1\relax}{\relax\detokenize{#1}\relax}{}{}
\usepackage{mathtools}
\usepackage{xcolor}
\usepackage{framed}
\allowdisplaybreaks
\numberwithin{equation}{section}
\usepackage[none]{hyphenat}
\usepackage{graphics}
\usepackage{graphicx}
\usepackage{epstopdf}
\usepackage{floatrow}
\usepackage{caption}
\begin{document}
\title{\bf CONTROLLABILITY OF NETWORKED SYSTEM WITH HETEROGENEOUS DYNAMICS}
\author {Abhijith Ajayakumar\hspace*{0.1 cm} and \hspace*{0.1 cm}Raju K. George\\
 Department of Mathematics,  Indian Institute of Space Science and Technology,
 Thiruvananthapuram 695 547, India.\\
\vspace*{0.15 cm}
Corresponding author: abhijithajayakumar.19@res.iist.ac.in}
\date{\today}
\maketitle{}

\begin{abstract}
\noindent In this paper, a necessary and sufficient condition for the controllability of networked systems with heterogeneous dynamics is established where the nodes are higher dimensional linear time invariant(LTI) systems and the network topology is directed and weighted. The controllability of networked system over some specific topologies are also examined and some non-controllability results are obtained. The theoretical results are demonstrated with examples.
\vspace{0.1 in}
\vspace{0.1 in}
\end{abstract}

\section{Introduction}\label{section1}
Controllability of dynamical systems  is one of the key concepts in mathematical control theory introduced by R.E Kalman.(\cite{kalman1960general}) State controllability  expresses how ably we can guide a dynamical system to a desired final state from an arbitrary initial state within a finite period of time.(\cite{kalman1962canonical}) Structural controllability introduced by Lin(\cite{lin1974structural}), emphasizes the unparalleled role of system structure in determining the controllability of the  state.   The idea of controllability, whether it is state or structural has been extensively studied for various types of systems and conditions for controllability are obtained for such systems over the past few decades.(\cite{hautus1969controllability,glover1976characterization,mayeda1981structural,
tarokh1992measures,jarczyk2011strong}) Most of these criteria focus on the dynamics of  a single higher dimensional system. However, in the real world, the situation of networked control systems are comparatively much larger than that of the single stand alone control systems. In general, modelling of complex natural and technological systems require a collection of individual systems together with an inter connection topology.  (\cite{farhangi2009path,muller2011few,gu2015controllability,bassett2017network}) Controllability of a large scale complex networked systems give rise to fascinating challenges due to computational burden and  various aspects of the systems like structural complexity, node dynamics , interaction amongst various nodes etc.The question of controllability of networked systems is of utmost importance to the developments in numerous branches of science and engineering. \par 

Numerous principles were formulated to study a networked system's  controllability  over the years which employ different types of graphic properties and  matrix rank conditions .The problem of controllability of inter connected systems date back to the work of Gilbert (\cite{gilbert1963controllability}) in which the controllability and observability conditions of systems connected in parallel and series interconnections are studied followed by the works of Callier et al. (\cite{callier1975necessary}) and Fuhrmann.(\cite{fuhrmann1975controllability}) The interconnection structures are more complex than those considered in these works and hence needed the idea of weighted directed graphs to represent the network topology. By dividing the nodes into leaders and followers, some conditions on network topology were derived by Tanner (\cite{tanner2004controllability}) which ensured the controllability of a group of nodes with a single leader. In the work of Hara et al.(\cite{hara2009lti}) networks in which each node is a copy of the same single-input-single-output(SISO) system were considered and a necessary and sufficient condition for the controllability and observability was obtained. Zhou \cite{zhou2015controllability} established a necessary and sufficient condition for the controllability and observability of a heterogeneous networked system where certain transfer matrices associated with the system have full column rank.  The controllability problem of a networked  multi-input-multi-output was addressed by Wang et al.(\cite{wang2016controllability}) and a necessary and sufficient condition for controllability of a homogeneous system was obtained. Based on the above work, Wang et al.(\cite{wang2017controllability}) further derived a necessary and sufficient condition for the state controllability of  a homogeneous linear time invariant system where communications are performed through one dimensional connections. The controllability of a homogeneous networked system over some special network topologies such as trees,cycles etc were also discussed. A necessary and sufficient condition for the controllability of a multi-input-multi-output(MIMO) homogeneous linear time invariant system with directed, weighted network topology  and higher dimensional node dynamics is derived by Hao et al.\cite{hao2018further} Compared to the work of Wang et al. (\cite{wang2016controllability}) Hao's result is easy to verify as it does not require solving matrix equations. A necessary and sufficient condition for the controllability of a heterogeneous networked system was derived from Popov - Belevitch- Hautus (PBH)(\cite{terrell2009stability}) rank condition by Wang et al.(\cite{peiruwang2017controllability}) The connection between the controllability of the networked topology and controllability of the whole networked system is discussed in detail by Xiang et. al (\cite{xiang2019controllability}) and a necessary and sufficient condition for the controllability of a heterogeneous system is obtained in terms of some rank conditions. The notion of structural controllability of large scale networked systems is also studied extensively.(\cite{zamani2009structural,blackhall2010structural,chapman2013strong,xue2019structural}) Studies on graph theoretic notions on system theoretic properties are also under discussion. (\cite{ji2007graph,rahmani2009controllability,nepusz2012controlling,liu2013graph,yaziciouglu2016graph})  A brief survey of recent advances in the study of the controllability of networked linear dynamical systems, regarding the relationship of the network topology, node dynamics , external control inputs and inner dynamical interactions with controllability of such complex networked systems is shown in Xiang et. al.(\cite{xiang2019advances}) \par

Based the above analysis, most of the available results are for homogeneous linear time invariant networked systems. In this paper a necessary and sufficient condition for the controllability of a heterogeneous system model is obtained. Our result generalizes the work of Hao et. al(\cite{hao2018further}) which was for the controllability of a homogeneous linear time invariant networked systems and this enables us to consider a larger class of systems. Compared to the result in the work  Xiang et. al(\cite{xiang2019controllability}), the condition in this paper does not require solving matrix equations and can be easily verified than the result by Wang et. al(\cite{wang2016controllability}). The paper is organised as follows. Some preliminaries are given in  section 2 . The  controllability problem is formulated in section 3. In section 4, a necessary and sufficient condition is obtained for the controllability of the heterogeneous networked system formulated in section 3  and some controllability results of the networked system over some specific topologies are developed. The derived results are substantiated with examples. Conclusion and some remarks are given in section 5.

\section{Preliminaries}\label{section2}
Throughout, let $\mathbb{R}$ denote the field of real numbers, $\mathbb{R}^n$ denote the vector space of real $n-$vectors and $\mathbb{R}^{m \times n}$ denotes the set of $m \times n$ real matrices. Let $I_N$ be the identity matrix of order $N$, $e_i$ be the row vector with the $i^{th}$ entry 1 and all other entries zero, $diag \lbrace a_1,a_2, \ldots ,a_n \rbrace $ denotes diagonal matrix of order $n \times n$ with diagonal entries $ a_1,a_2, \ldots ,a_n $ and $uppertriang \lbrace a_1,a_2, \ldots ,a_n \rbrace $ denote the upper triangular matrix with diagonal entries $ a_1,a_2, \ldots ,a_n $. For matrices $A_1,A_2, \ldots ,A_n$, $blockdiag \lbrace A_1,A_2, \ldots A_n \rbrace $ denotes the block diagonal matrix with $A_1,A_2, \ldots ,A_n$ as diagonal entries and $blockuppertriang \lbrace A_1,A_2, \ldots A_n \rbrace $ denotes the block upper triangular matrix with $A_1,A_2, \ldots ,A_n$ as diagonal entries. The Kronecker product of matrices $A$ and $B$ is denoted by $A \otimes B$. Let $\sigma(A)$ denote the eigen spectrum of the matrix $A$.
\begin{lemma}( \cite{rugh1996linear}) \label{le1}
A linear time invariant control system characterized by the pair of matrices $(A,B)$ is controllable if and only if $\nu A=\lambda \nu$ implies that $\nu B \neq 0$ where $\nu $ is the non-zero left eigenvector of $A$ associated with the eigenvalue $\lambda $. 
\end{lemma} 
\begin{lemma} (\cite{horn2012matrix}) \label{le2}
Let $A,B,C$ and $D$ be matrices with appropriate order, then \begin{enumerate}
\item $(A \otimes B)(C \otimes D)=(AC \otimes BD)$
\item $(A \otimes B)^{-1}=A^{-1} \otimes B^{-1}$ if $A$ and $B$ are invertible.
\item $(A+B) \otimes C =A \otimes C + B \otimes C$
\item $A \otimes (B+C)= A \otimes B + A \otimes C$
\item $A \otimes B =0$ if and only if $A=0$ or $B=0$
\end{enumerate}
\end{lemma}
\begin{lemma} (\cite{horn1994topics}) \label{le3}
Suppose that two matrices $A$ and $B$ are similar. i.e, there exists a non-singular matrix $P$ such that $PBP^{-1}=A$. If $\nu $ is a left eigenvector of $A$ with respect to the eigenvalue $\lambda $, then $\nu P$ is an eigenvector of $B$ with respect to the eigenvalue $\lambda $.
\end{lemma}
\section{Model formulation}\label{section3}
Consider a networked linear time invariant system with $N$ nodes, where each node system is of dimension, $n$. Specifically, the dynamical system corresponding to the node $i$ is described by 
\begin{equation}
\dot{x_i}(t)=A_ix_i(t)+\sum_{j=1}^{N}c_{ij}Hx_j(t)+d_iBu_i(t),\ \ \ i=1,2,\ldots ,N \label{eq:1}
\end{equation} 
where $x_i(t)\in \mathbb{R}^n$ is a state vector; $u_i(t) \in \mathbb{R}^m $ is an external control input vector; $A_i\in \mathbb{R}^{n \times n}$ is the state matrix of node $i$; $ B \in \mathbb{R}^{n \times m}$ is the input matrix, with $d_i=1$ if node $i$ is under control, otherwise $d_i=0$. $c_{ij} \in \mathbb{R}$ represents the coupling strength between the nodes $i$ and $j$ with   $c_{ij} \neq 0$ if there is a communication from node $j$ to node $i$, but otherwise $c_{ij}=0$, for all $i,j=1,2,\ldots ,N $ and $H\in \mathbb{R}^{n\times n}$ is the inner coupling matrix describing the interconnections among the components $x_j,j=1,2, \ldots ,N$.\\
Denote \begin{equation}
C=\left[ c_{ij} \right] \in \mathbb{R}^{N \times N}\ \ \ and \ \ \ D=diag \lbrace d_1,d_2, \ldots ,d_N \rbrace \label{eq:2}
\end{equation}
which represent the network topology and external input channels of the networked system \eqref{eq:1}, respectively. Denote the whole state of the networked system by $X=\left[ x_1^T,\ldots ,x_N^T \right]^T $ and  the total external control input vector by $U=\left[ u_1^T,\ldots ,u_N^T \right] ^T$.
Then the networked system \eqref{eq:1} can be rewritten in a compact form as
\begin{equation} 
\dot{X}(t)=F X(t)+ GU(t) \label{eq:3}
\end{equation}
with \begin{equation}
  F=A+C\otimes H ,\  G=D\otimes B \label{eq:4}
  \end{equation}
 where $A=blockdiag\lbrace A_1,A_2,\ldots ,A_N \rbrace$.
\section{Main Results}\label{sec4}
 \subsection{Controllability in a General Network Topology}\label{subsection1}
 In this section a necessary and sufficient condition for the controllability of \eqref{eq:3} is derived. First of all, the left eigenvectors of $F$ are expressed in terms of eigenvectors of some smaller matrices.
 \begin{theorem}\label{prop1}
Let $T$ be a matrix such that $TCT^{-1}=J=uppertriang\lbrace \lambda_1,\lambda_2,\ldots ,\lambda_N \rbrace $, where $J$ is the Jordan Canonical Form of $C$. Suppose that $T \otimes I$ commutes with $A$. Also let $M_i=\lbrace \mu_i^1, \ldots ,\mu_i^{q_i} \rbrace$, be the set of eigenvalues of $A_i + \lambda_i H ,i=1,2,\ldots ,N$. Then $\sigma (F) =\lbrace \mu_1^1, \ldots ,\mu_1^{q_1},\ldots ,\mu_N^1, \ldots ,\mu_N^{q_N} \rbrace $. Let $\xi_{ij}^k,k=1,\ldots , \gamma_{ij}$ be the left eigenvectors of $A_i+ \lambda_iH$ corresponding to $\mu_i^j,j=1,\ldots ,q_i ,i=1,\ldots ,N$, where $\gamma_{ij} \geq 1$ is the geometric multiplicity of the eigenvalue $\mu_i^j$ for $A_i+\lambda_iH$. If $J$ is a diagonal matrix, $e_iT \otimes \xi_{ij}^k,k=1,\ldots , \gamma_{ij}$ are the left eigenvectors of $F$ corresponding to the eigenvalue $\mu_i^j,j=1,\ldots ,q_i ,i=1,\ldots ,N$ . If $J$ contains a Jordan block of order $l\geq 2$ for some eigenvalue $\lambda_{i_0}$ of $C$ with $\xi_{ij}^kH=0$ for all $i=i_0,i_0+1,\ldots ,i_0+l-1,j=1,2,\ldots,q_i, k=1,2,\ldots , \gamma_{ij}$, then $e_{i}T \otimes \xi_{ij}^k,k=1,\ldots , \gamma_{ij}$ are the left eigenvectors of $F$ corresponding to the eigenvalue $\mu_{i}^{j}, i=1,2,\ldots ,N,j=1,2,\ldots ,q_i$.
 \end{theorem}
 \begin{proof}
 Let $T$ be the non-singular matrix such that $TCT^{-1}=J$, where $J=uppertriang\lbrace \lambda_1,\lambda_2,\ldots ,\lambda_N \rbrace$  is the Jordan Canonical form of $C$. Then, we have 
\begin{align*}
\tilde{F} &=(T \otimes I)(A+C \otimes  H)(T^{-1} \otimes I) \\
&= A+J \otimes H \\
&= A+uppertriang\lbrace \lambda_1,\lambda_2,\ldots ,\lambda_N \rbrace \otimes H \\
&= blockuppertriang \lbrace A_1+\lambda_1 H, \cdots ,A_N +\lambda_N H\rbrace
\end{align*}
Since $\tilde{F}$ and $F$ have same eigenvalues, $\sigma (F) =\lbrace \mu_1^1, \ldots ,\mu_1^{q_1},\ldots ,\mu_N^1, \ldots ,\mu_N^{q_N} \rbrace $. Now let $\xi_{ij}^k,k=1,\ldots , \gamma_{ij}$ are the left eigenvectors of $A_i+ \lambda_iH$ corresponding to $\mu_i^j,j=1,\ldots ,q_i ,i=1,\ldots ,N$. If $J$ is a diagonal matrix, $\tilde{F}$ is a block diagonal matrix and hence $e_i \otimes \xi_{ij}^k,k=1,\ldots , \gamma_{ij}$ are left eigenvectors of $\tilde{F} $ corresponding to $\mu_i^j,j=1,\ldots ,q_i ,i=1,\ldots ,N$. Now suppose that $J$ contains a Jordan block of order 2, corresponding to the eigenvalue $\lambda_{i_0}$ of $C$. Then the matrix $\tilde{F}$ contains block matrix of the form 
\begin{align}\label{4.1}
\mathcal{A}=
\begin{bmatrix}
A_{i_0}+\lambda_{i_0}H & H \\
0 & A_{i_0+1}+ \lambda_{i_0+1}H
\end{bmatrix}
\end{align}
 Clearly, $e_{i_0+1} \otimes \xi_{i_0+1j}^{k},k=1,2, \ldots , \gamma_{i_0+1j}$ are eigenvectors of $\tilde{F}$ corresponding to the eigenvalues $\mu_{i_0+1j},j=1,2,\ldots ,q_{i_0+1}$. If $\xi_{i_0j_0}^kH=0$ for all $k=1,2,\ldots , \gamma_{i_0j_0}$,  $e_{i_0}\otimes \xi_{i_0j_0}^k,k=1,2,\ldots , \gamma_{i_0j_0}$ are  left eigenvectors of $\tilde{F}$ corresponding to the eigenvalue $\mu_{i_0j_0}$. Now suppose that $J$ contains a Jordan block order $l\geq 2$ for some eigenvalue $\lambda_{i_0}$ of $C$,  then again we can consider $(l-1)$ block matrices of the form  \eqref{4.1} and by using the fact that $\xi_{ij}^kH=0$ for all $i=i_0,i_0+1,\ldots ,i_0+l-1,j=1,2,\ldots,q_i, k=1,2,\ldots , \gamma_{ij}$ we get $ e_i \otimes \xi_{ij}^k,k=1,2,\ldots , \gamma_{ij}$ are left eigenvectors of $\tilde{F}$ corresponding to the eigenvalue $\mu_i^j,i=1,2,\ldots ,N,j=1,2,\ldots ,q_i$. Now we will prove that, they are the only eigenvectors of $\tilde{F}$. Suppose that $\tilde{F}$ does not have any Jordan blocks and let $\xi = \begin{bmatrix}
 \xi_1 & \xi_2 & \ldots & \xi_N 
\end{bmatrix} \in \mathbb{R}^{Nn}$ be a left eigenvector of $\tilde{F}$ corresponding to the eigenvalue $\mu $, where $ \xi_1 , \xi_2 , \ldots , \xi_N 
 \in \mathbb{R}^{n}$. Then $\xi^T \tilde{F} = \mu \xi^T$ implies that
 $$  \begin{bmatrix}
 \xi_1 \left( A_1+\lambda_1 H \right)  \\
 \xi_2 \left( A_2+\lambda_2 H \right)  \\
\vdots  \\
 \xi_N \left( A_N+\lambda_N H \right) 
\end{bmatrix}^T = \mu \begin{bmatrix}
 \xi_1 \\ \xi_2 \\ \vdots \\ \xi_N 
\end{bmatrix}^T $$
which implies that $\mu $ is an eigenvalue of $A_i+\lambda_i H$ for all $i$ with $\xi_i$ as an eigenvector. Now suppose that $\tilde{F}$ has a block of the form \eqref{4.1}.  Then $\xi^T \tilde{F} = \mu \xi^T$ implies that 
 $$  \begin{bmatrix}
 \xi_1 \left( A_1+\lambda_1 H \right)  \\
\vdots  \\
 \xi_i \left( A_i+\lambda_i H \right)  \\
 \xi_iH+ \xi_{i+1}H  \left( A_2+\lambda_2 H \right)  \\
 \vdots \\
 \xi_1 \left( A_N+\lambda_N H \right) 
\end{bmatrix}^T = \mu \begin{bmatrix}
 \xi_1 \\ \vdots \\ \xi_i \\ \xi_{i+1} \\ \vdots \\ \xi_N 
\end{bmatrix}^T $$
As $\xi_i \left( A_i+\lambda_i H \right)= \mu \xi_i$, $\xi_i$ is a left eigenvector of $ A_i+\lambda_i H $. Then by our hypothesis, $\xi_iH=0$. Hence $\mu $ is an eigenvalue of $A_i+\lambda_i H$ for all $i$ with $\xi$ as an eigenvector. Thus if $A_i+\lambda_i H,i=1,2,\ldots ,N$ does not have a common eigenvalue, then the left eigenvectors of $\tilde{F}$ are of the form $e_i \otimes \xi$, where $\xi $ is a left eigenvector of $A_i + \lambda_i H$ for some $i$. If they have a common eigenvalue, the eigenvectors are either of the form  $e_i \otimes \xi$, where $\xi $ is a left eigenvector of $A_i + \lambda_i H$ for some $i$ or of the form $\sum_{\alpha =1}^r e_{i_{\alpha}} \otimes \xi_{i_{\alpha}} $, where $A_i + \lambda_i H,i \in \lbrace i_1,i_2, \ldots ,i_r \rbrace $ have a common eigenvalue $\mu$ with eigenvector $\xi_{i_{\alpha}}$ for each $i_1,i_2, \ldots ,i_r$.
  \par   Thus in both cases, by Lemma \ref{le2}(1) and Lemma \ref{le3}, $\left( e_i \otimes \xi_{ij}^k\right)\left( T \otimes I \right)=   e_iT \otimes \xi_{ij}^k(k=1,\ldots , \gamma_{ij})$ are the left eigenvectors of $F$ corresponding to $\mu_i^j,j=1,\ldots ,q_i ,i=1,\ldots ,N$.
 \end{proof}
 Now using the above theorem, we will prove the following necessary and sufficient condition for controllability of a heterogeneous networked system \eqref{eq:3}. 
 \begin{theorem} \label{thm1}
  Consider a heterogeneous networked system with a network topology $C$  where $TCT^{-1}=J=uppertriang\lbrace \lambda_1,\lambda_2,\ldots , \lambda_N \rbrace $. Suppose that $\xi_{ij}^kH=0$ for all $i=i_0,i_0+1,\ldots ,i_0+l-1,j=1,2,\ldots,q_i, k=1,2,\ldots , \gamma_{ij}$ if $J$ contains a Jordan block of order $l\geq 2$ corresponding to the eigenvalue $\lambda_{i_0}$ of $C$ , where $\xi_{ij}^k,i=1,2,\ldots ,N,j=1,2,\ldots ,q_i,k=1,2,\ldots ,\gamma_{ij}$ are the left eigenvectors of $A_i+\lambda_iH$ corresponding to the eigenvalues $\mu_i^j,i=1,2,\ldots ,N,j=1,2,\ldots ,q_i$. Then the networked system is controllable if and only if 
\begin{enumerate}
\item $e_iTD \neq 0$ for all $i=1,\ldots ,N$ , where $\lbrace e_i \rbrace$ is the canonical basis for $\mathbb{R}^N$.
\item $(A_i+\lambda_i H,B)$ is controllable, for $i=1,2,\cdots ,N$; and
\item If matrices $A_{i_1}+\lambda_{i_1}H ,\ldots , A_{i_p}+\lambda_{i_p}H (\lambda_{i_k} \in \Lambda,\ for\ k=1, \ldots ,p,\ p>1)$ have a common eigenvalue $\sigma $, then $(e_{i_1}TD) \otimes (\xi_{i_1}^1B), \cdots ,(e_{i_1}TD) \otimes (\xi_{i_1}^{\gamma_{i_1}}B), \ldots ,(e_{i_p}TD) \otimes (\xi_{i_p}^1B), \ldots ,(e_{i_p}TD) \otimes (\xi_{i_p}^{\gamma_{i_p}}B)$ are linearly independent where $\gamma_{i_k} \geq 1$ is the geometric multiplicity of $\sigma $ for $A_{i_k}+\lambda_{i_k}H;\xi_{i_k}^l(l=1, \ldots , \gamma_{i_k})$ are the left eigenvectors of $A_{i_k}+\lambda_{i_k}H$ corresponding to $\sigma , k=1,\ldots ,p$.
\end{enumerate}
 \end{theorem}
 \begin{proof}
 \textbf{(Necessary part)} From Theorem \eqref{prop1} we have that,$e_iT \otimes \xi_{ij}^k(k=1,\ldots , \gamma_{ij})$ are  left eigenvectors of $F$ corresponding to $\mu_i^j,j=1,\ldots ,q_i ,i=1,\ldots ,N$. If the networked system \eqref{eq:3} is controllable, then $$(e_iT \otimes \xi_{ij}^l)(D \otimes B)\neq 0,\ for\ l=1,\ldots , \gamma_{ij},j=1,\ldots ,q_i ,i=1,\ldots , N$$ which implies that $$e_iTD \neq 0,\ i=1,\ldots ,N,$$ and $$\xi_{ij}^lB \neq 0,\ for \  l=1,\ldots , \gamma_{ij},j=1,\ldots ,q_i ,i=1,\ldots , N$$
Since $\xi_{ij}^l $ is an arbitrary left eigenvector of $A_i +\lambda_i H$, this implies that $(A_i +\lambda_i H,B)$ is controllable, for $i=1,\ldots ,N$.\\
Assume that the matrices $A_{i_1}+\lambda_{i_1}H ,\cdots , A_{i_p}+\lambda_{i_p}H (\lambda_{i_k} \in \Lambda,\ for\ k=1, \ldots ,p,\ p>1)$ have a common eigenvalue $\sigma $. The geometric multiplicity of $\sigma $ for $A_{i_k}+\lambda_{i_k}H$ is denoted by $\gamma_{i_k}$, and the corresponding eigenvectors are denoted as $\xi_{i_k}^1,\ldots ,\xi_{i_k}^{\gamma_{i_k}},$ where $k=1, \ldots ,p$. Then all the left eigenvectors of $F$ corresponding to $\sigma $ can be expressed in the form of $\sum_{k=1}^p\sum_{l=1}^{\gamma_{i_k}}\alpha_{kl}(e_{i_k}T \otimes \xi_{i_k}^l)$, where $\alpha_{kl} \in \mathbb{R}(k=1,\ldots ,p,l=1,\ldots ,\gamma_{i_k})$ are parameters, which are not all zero. If the networked system is controllable, then $$\left[ \sum_{k=1}^p\sum_{l=1}^{\gamma_{i_k}}\alpha_{kl}(e_{i_k}T \otimes \xi_{i_k}^l)\right] (D \otimes B) \neq 0$$
Consequently, we have
$$\sum_{k=1}^p\sum_{l=1}^{\gamma_{ik}}\alpha_{kl}(e_{i_k}TD)\otimes (\xi_{ik}^lB)\neq 0$$
for any scalars $\alpha_{kl} \in \mathbb{R}(k=1,\ldots ,p,l=1,\ldots ,\gamma_{ik})$ , which are not all zero. Therefore, $(e_{i1}TD) \otimes (\xi_{i1}^1B), \ldots ,(e_{i1}TD) \otimes (\xi_{i1}^{\gamma_{i1}}B), \ldots ,(e_{ip}TD) \otimes (\xi_{ip}^1B), \ldots ,(e_{ip}TD) \otimes (\xi_{ip}^{\gamma_{ip}}B)$ are linearly independent.\\
\textbf{(Sufficiency part)} Suppose that the networked system is uncontrollable, then we will prove that atleast one condition in Theorem 1 does not hold. If the networked system is not controllable, then there exists a left eigenpair of $F$, denoted as $(\tilde{\mu},\tilde{v})$, such that $\tilde{v}G=0$.
\begin{itemize}
\item If $\tilde{\mu} \in M_{i_0}$ and $\tilde{\mu} \notin M_1 \cup \ldots \cup M_{i_0-1}\cup M_{i_0+1}\cup \ldots \cup M_N$. Then as earlier, $\tilde{v}=\sum_{l=1}^{\gamma_{{i_0}_{j_0}}}\alpha_0^l(e_{i_0}T\otimes \xi_{{i_0}_{j_0}}^l)$, where $\xi_{{i_0}_{j_0}}^1,\ldots ,\xi_{{i_0}_{j_0}}^{\gamma_{{i_0}_{j_0}}}$ are the left eigenvectors of $A_{i_0}+\lambda_{i_0}H$ corresponding to $\tilde{\mu}$; $\left[ \alpha_0^1,\ldots ,\alpha_0^{\gamma_{{i_0}_{j_0}}} \right] $ is some non zero vector. Now $\tilde{v}G=0$ implies
\begin{align*}
\sum_{l=1}^{\gamma_{{i_0}_{j_0}}}\alpha_0^l(e_{i_0}T\otimes \xi_{{i_0}_{j_0}}^l)(D \otimes B)&=\sum_{l=1}^{\gamma_{{i_0}_{j_0}}}\alpha_0^l(e_{i_0}TD)\otimes (\xi_{{i_0}_{j_0}}^lB)\\
&=(e_{i_0}TD) \otimes \left( \sum_{l=1}^{\gamma_{{i_0}_{j_0}}}\alpha_0^l\xi_{{i_0}_{j_0}}^lB \right)=0 
\end{align*}
which implies $e_{i_0}TD=0$ or $ \sum_{l=1}^{\gamma_{{i_0}_{j_0}}}\alpha_0^l\xi_{{i_0}_{j_0}}^lB=0$. Since $ \sum_{l=1}^{\gamma_{{i_0}_{j_0}}}\alpha_0^l\xi_{{i_0}_{j_0}}^l$ is the left eigenvector of $A_{i_0}+ \lambda_{i_0}H$, if $ \sum_{l=1}^{\gamma_{{i_0}_{j_0}}}\alpha_0^l\xi_{{i_0}_{j_0}}^lB=0$, then $\left(A_{i_0}+ \lambda_{i_0}H,B \right) $ is uncontrollable. i.e,  If the networked system is uncontrollable, then either there exists $\lambda_{i_0} \in \Lambda$ such that $\left(A_{i_0}+ \lambda_{i_0}H,B \right) $ is uncontrollable or $e_{i_0}TD=0$ for some $i_0$.
\item If $\tilde{\mu}$ is the common eigenvalue of the matrices $A_{i_1}+\lambda_{i_1}H,\ldots ,A_{i_p}+\lambda_{i_p}H(\lambda_{i_k} \in \Lambda,\ for\ k=1, \ldots ,p, p>1)$. The geometric multiplicity of $\tilde{\mu}$ for $A_{i_k}+\lambda_{i_k}H$ is denoted as $\gamma_{i_k}$, and the corresponding eigenvectors are denoted as $\xi_{i_k}^1,\ldots ,\xi_{i_k}^{\gamma_{i_k}}$, where $k=1,\ldots ,p$. Since $\tilde{v}$ can be expressed in the form $\sum_{k=1}^p\sum_{l=1}^{\gamma_{i_k}}\alpha_0^{kl}\left( e_{i_k}T \otimes \xi_{ik}^l\right) $, where $\alpha_0^{kl}(l=1,\ldots ,\gamma_{i_k},k=1,\ldots ,p)$ are some scalars, which are not all zero. Then $\tilde{v}G=0$ implies that there exists a non zero vector $\left[ \alpha_0^{11}, \ldots ,\alpha_0^{1\gamma_{i_1}}, \ldots ,\alpha_0^{p1}, \ldots ,\alpha_0^{p\gamma_{i_p}}\right] $ such that
$$\left[ \sum_{k=1}^p\sum_{l=1}^{\gamma_{i_k}}\alpha_0^{kl}\left( e_{i_k}T \otimes \xi_{i_k}^l\right) \right] (D\otimes B)=\sum_{k=1}^p\sum_{l=1}^{\gamma_{i_k}}\alpha_0^{kl}\left[ (e_{i_k}TD) \otimes (\xi_{i_k}^lB) \right]=0 $$ which implies that $(e_{i_1}TD) \otimes (\xi_{i_1}^1B), \ldots ,(e_{i_1}TD) \otimes (\xi_{i_1}^{\gamma_{i_1}}B), \ldots ,(e_{i_p}TD) \otimes (\xi_{i_p}^1B), \ldots ,(e_{i_p}TD) \otimes (\xi_{i_p}^{\gamma_{i_p}}B)$ are linearly dependent.
\end{itemize}
Therefore, if the networked system is uncontrollable, then atleast one condition in Theorem \ref{thm1} does not hold.
 \end{proof}
 The following examples establish the efficiency of the result for a heterogeneous networked system.
 \begin{example}
 Consider a system composed of 3 nodes in which two nodes are identical ,\vspace*{0.15 cm}  $A_1=A_3=\begin{bmatrix}
1 & 0 & 0 \\
0 & 1 & 1  \\
0 & 1 & 1
\end{bmatrix},A_2=\begin{bmatrix}
1 & 1 & 0 \\
0 & 1 & 1  \\
-1 & 0 & 1
\end{bmatrix},B=\begin{bmatrix}
1 \\
2 \\
1
\end{bmatrix}, H=\begin{bmatrix}
0 & 0 & 1 \\
0 & 0 & 0 \\
1 & 0 & 0
\end{bmatrix}, C=\begin{bmatrix}
0 & 0 & 1 \\
0 & 1 & 1 \\
0 & 0 & 1
\end{bmatrix}$ and $D=\begin{bmatrix}
1 & 0 & 0 \\
0 & 1 & 0 \\
0 & 0 & 1
\end{bmatrix}$. Then,
\begin{itemize}
\item  there exists $T=\begin{bmatrix}
1 & 0 & -1 \\
0 & 1 & 0 \\
0 & 0 & 1
\end{bmatrix}$ such that $TCT^{-1}=J$ where $J=\begin{bmatrix}
0 & 0 & 0 \\
0 & 1 & 1 \\
0 & 0 & 1
\end{bmatrix}$. Here $\lambda_1=0 ,\lambda_2=1$ and $\lambda_3=1$. Clearly $J$ contains a Jordan block of order 2. Observe that $\xi_2^1 = \begin{bmatrix}
0 & 1 & 0 
\end{bmatrix} $ is the only left eigenvector corresponding to $A_2+H$ and $\xi_2^1H=0$.  Also $T \otimes I$ commutes with $A$.
\item  $e_iTD \neq 0$ for all $i=1,2,3$ .
\item $(A_1,B),(A_2+H,B)$ and $(A_3+H,B)$ are controllable.
\item $A_1$,$A_2+H$ and $A_3+H$ have $\sigma =1 $ as a common eigenvalue with corresponding left eigenvectors $\xi_1^1 = \begin{bmatrix}
1 & 0 & 0 
\end{bmatrix} $,$\xi_2^1 = \begin{bmatrix}
0 & 1 & 0 
\end{bmatrix} $ and $\xi_3^1 = \begin{bmatrix}
1 & -1 & 0 
\end{bmatrix} $. Now $e_1TD \otimes \xi_1^1 B=\begin{bmatrix}
1 & 0 & -1 
\end{bmatrix} $,$e_2TD \otimes \xi_2^1 B=\begin{bmatrix}
0 & 1 & 0
\end{bmatrix} $ and $e_3TD \otimes \xi_3^1 B=\begin{bmatrix}
0 & 0 & -1
\end{bmatrix} $ are linearly independent.
\end{itemize}
 Thus all the conditions of Theorem \ref{thm1} are verified and hence the system is controllable. The controllability of the above system can be verified by Kalman's rank condition.
 \end{example}
 \begin{example}
 Consider a system composed of 3 nodes in which two nodes are identical ,\vspace*{0.15 cm} $A_1=\begin{bmatrix}
0 & 1 & 0 \\
0 & 0 & 1  \\
0 & 0 & 1
\end{bmatrix},A_2=A_3=\begin{bmatrix}
0 & 1 & 0 \\
0 & 0 & 1  \\
0 & 0 & 0
\end{bmatrix},B=\begin{bmatrix}
1 \\
0 \\
1
\end{bmatrix}, H=\begin{bmatrix}
1 & 0 & 0 \\
0 & 1 & 0 \\
0 & 0 & 1
\end{bmatrix}, C=\begin{bmatrix}
1 & 0 & 0 \\
0 & 1 & 0 \\
0 & 1 & 0
\end{bmatrix}$ and $D=\begin{bmatrix}
1 & 0 & 0 \\
0 & 1 & 0 \\
0 & 0 & 0
\end{bmatrix}$. Then,
\begin{itemize}
\item  there exists $T=\begin{bmatrix}
1 & 0 & 0 \\
0 & -1 & 1 \\
0 & \sqrt{2} & 0
\end{bmatrix}$ such that $TCT^{-1}=\begin{bmatrix}
1 & 0 & 0 \\
0 & 0 & 0 \\
0 & 0 & 1
\end{bmatrix}=J$. Here $\lambda_1=1 ,\lambda_2=0$ and $\lambda_3=1$. Clearly $J$ does not have any Jordan block of order $\geq 2$.  Also $T \otimes I$ commutes with $A$
\item $e_iTD \neq 0$ for all $i=1,2,3$.
\item $(A_1+H,B),(A_2,B)$ and $(A_3+H,B)$ are controllable.
\item $A_1+H$ and $A_3+H$ have a common eigenvalue 1 with corresponding left eigenvectors $\xi_1^1 = \begin{bmatrix}
0 & 1 & -1 
\end{bmatrix} $ and $\xi_3^1 = \begin{bmatrix}
0 & 0 & 1 
\end{bmatrix} $. Now $e_1TD \otimes \xi_1^1 B=\begin{bmatrix}
-1 & 0 & 0 
\end{bmatrix} $ and $e_3TD \otimes \xi_3^1 B=\begin{bmatrix}
0 & \sqrt{2} & 0
\end{bmatrix} $ are linearly independent.
\end{itemize}    
 Thus all the conditions of Theorem \ref{thm1} are verified and hence the system is controllable. The controllability of the above system can be verified by Kalman's rank condition.
 \end{example}
 We now state the following result as a corollary of the above theorem in which we provide a situation where the system is not controllable.
  \begin{corollary} \label{co1}
  If $e_iTD=0$ for some $i$, then the given system is not controllable. 
  \end{corollary}
  \begin{proof}
 We have, $e_iT \otimes \xi_{ij}^k(k=1,\cdots , \gamma_{ij})$  are  left eigenvectors of $F$ corresponding to $\mu_i^j,j=1,\ldots ,q_i ,i=1,\ldots ,N$. If $e_iTD=0$ for some $i$, say $i_0$, then $(e_{i_0}T \otimes \xi_{{i_0}_j}^k)(D \otimes B)=0$ for all $j=1,2,\ldots ,q_{i_0},k=1,2,\ldots , \gamma_{i_0j}$ which implies that the given system is not controllable.
  \end{proof}
  \begin{example}
  Consider a network composed of 3 identical nodes,\vspace*{0.15 cm}  $A_1=A_2=A_3=
  \begin{bmatrix}
0 & 1 & 0 \\
0 & 0 & 1 \\
0 & 0 & 0
\end{bmatrix}, B=\begin{bmatrix}
1 \\
1 \\
1
\end{bmatrix}, H=\begin{bmatrix}
1 & 0 & 0 \\
0 & 1 & 0 \\
0 & 0 & 1
\end{bmatrix}, C=\begin{bmatrix}
0 & 0 & 1 \\
1 & 0 & 0 \\
1 & 0 & 0
\end{bmatrix}$, and $D= \begin{bmatrix}
0 & 0 & 0 \\
0 & 1 & 0 \\
0 & 0 & 0
\end{bmatrix}$. Then there exists $T= \begin{bmatrix}
0 & 1 & -1 \\
\frac{\sqrt{3}}{2} & 0 & \frac{\sqrt{3}}{2} \\
-\frac{\sqrt{3}}{2}& 0 & \frac{\sqrt{3}}{2}
\end{bmatrix}$ such that $TCT^{-1}=\begin{bmatrix}
0 & 0 & 0 \\
0 & 1 & 0 \\
0 & 0 & -1
\end{bmatrix}=J$. From Corollary \ref{co1} it is easy to verify that the networked system is not controllable as $e_2TD=0$.
\end{example}
When the networked system is homogeneous, we have the following result.
  \begin{theorem} \label{thm2}
 Consider a homogeneous networked system, That is, $A_i=\tilde{A}\ for\ all\ i=1,\ldots ,N$ with a network topology  $C$ such that  $TCT^{-1}=J=uppertriang\lbrace \lambda_1,\lambda_2,\ldots ,\lambda_N \rbrace$, where $J$ is the Jordan Canonical Form of $C$. Suppose that $\xi_{ij}^kH=0$ for all $i=i_0,i_0+1,\ldots ,i_0+l-1,j=1,2,\ldots,q_i, k=1,2,\ldots , \gamma_{ij}$ if $J$ contains a Jordan block of order $l\geq 2$ corresponding to the eigenvalue $\lambda_{i_0}$ of $C$ , where $\xi_{ij}^k,i=1,2,\ldots,k=1,2,\ldots ,\gamma_{ij}$ are the left eigenvectors of $A_i+\lambda_iH$ corresponding to the eigenvalues $\mu_i^j$ and $\gamma_{ij} \geq 1$ represents the geometric multiplicity of $\mu_i^j$. Then the networked system \eqref{eq:3} is controllable if and only if 
\begin{enumerate}
\item $e_iTD \neq 0$ for all $i=1, \ldots ,N$ , where $\lbrace e_i \rbrace$ is the canonical basis for $\mathbb{R}^N$.
\item $(\tilde{A}+\lambda_i H,B)$ is controllable, for $i=1,2,\cdots ,N$; and
\item If matrices $\tilde{A}+\lambda_{i_1}H ,\ldots , \tilde{A}+\lambda_{i_p}H (\lambda_{i_k} \in \Lambda,\ for\ k=1, \ldots ,p,\ p>1)$ have a common eigenvalue $\sigma $, then $(e_{i_1}TD) \otimes (\xi_{i_1}^1B), \ldots ,(e_{i_1}TD) \otimes (\xi_{i_1}^{\gamma_{i_1}}B), \ldots ,(e_{i_p}TD) \otimes (\xi_{i_p}^1B), \ldots ,(e_{i_p}TD) \otimes (\xi_{i_p}^{\gamma_{i_p}}B)$ are linearly independent  where $\gamma_{i_k} \geq 1$ is the geometric multiplicity of $\sigma $ for $\tilde{A}+\lambda_{i_k}H;\xi_{i_k}^l(l=1, \ldots , \gamma_{i_k})$ are the left eigenvectors of $\tilde{A}+\lambda_{i_k}H$ corresponding to $\sigma , k=1,\ldots ,p$.
\end{enumerate}
 \end{theorem}
 \begin{proof}
  If the given system is homogeneous, then it can be represented in the form $\dot{X}(t)=F X(t)+ GU(t)$ where $F=I \otimes \tilde{A}+C \otimes H$ and $G=D \otimes B$. Since $(T \otimes I)(I \otimes \tilde{A})=T \otimes \tilde{A}=(I \otimes \tilde{A})(T \otimes I)$, the result follows from Theorem \ref{thm1}.
 \end{proof}
  In the following example, we verify the conditions of the Theorem \ref{thm2} to obtain the controllability of a homogeneous networked system.
 \begin{example}
 Consider a system with two identical nodes, $A_1=A_2=\begin{bmatrix}
1 & 1  \\
0 & 1  \\
\end{bmatrix},B=\begin{bmatrix}
0 \\
1
\end{bmatrix}, H=\begin{bmatrix}
1 & 0 \\
0 & 0
\end{bmatrix}, C=\begin{bmatrix}
0 & 1 \\
1 & 0
\end{bmatrix}$ and $D=\begin{bmatrix}
1 & 0 \\
0 & 0
\end{bmatrix}$.
Then,
\begin{itemize}
\item  there exists $T=\begin{bmatrix}
-1 & 1 \\
1 & 1
\end{bmatrix}$ such that $TCT^{-1}=\begin{bmatrix}
-1 & 0 \\
0 & 1
\end{bmatrix}$. Here $\lambda_1=-1$ and $\lambda_2=1$.
\item $e_iTD \neq 0$ for all $i=1,2$
\item  $(A_1-H,B),(A_2+H,B)$ are controllable. As $A_1-H$ and $A_2+H$ does not have a common eigenvalue, condition 3 does not apply.
\end{itemize}
 Thus all the conditions of Theorem \ref{thm2} are verified and hence the system is controllable.
  \end{example}
 Hao et. al(\cite{hao2018further}) has given a necessary and sufficient condition for a homogeneous networked system with a diagonalizable network topology matrix. We can derive Hao's result as a corollary of Theorem \ref{thm2} as given below.
  \begin{corollary} (\cite{hao2018further})
If the given system is homogeneous with $A_i=A$, for $i=1,2,\cdots N$  and $C$ is diagonalizable with eigenvalues $\lambda_1 , \lambda_2, \ldots , \lambda_N$. Denote  $\Lambda = \lbrace \lambda_1 , \lambda_2, \ldots , \lambda_N \rbrace $. Then the networked system \eqref{eq:1} is controllable if and only if
\begin{enumerate}
\item $(C,D)$ is controllable;
\item $(A+\lambda_i H,B)$ is controllable, for $i=1,2,\cdots ,N$; and
\item If matrices $A+\lambda_{i_1}H ,\ldots , A+\lambda_{i_p}H \left(  \lambda_{i_k} \in \Lambda,\ for\ k=1, \ldots ,p,\ p>1\right) $ have a common eigenvalue $\sigma $, then $(t_{i_1}D) \otimes (\xi_{i_1}^1B), \ldots ,(t_{i_1}D) \otimes (\xi_{i_1}^{\gamma_{i_1}}B), \ldots ,(t_{i_p}D) \otimes (\xi_{i_p}^1B), \ldots ,(t_{i_p}D) \otimes (\xi_{i_p}^{\gamma_{i_p}}B)$ are linearly independent , where $t_{i_k}$ is the left eigenvector of $C$ corresponding to the eigenvalue $\lambda_{i_k}$; $\gamma_{i_k} \geq 1$ is the geometric multiplicity of $\sigma $ for $A+\lambda_{i_k}H;\xi_{i_k}^l(l=1, \ldots , \gamma_{i_k})$ are the left eigenvectors of $A+\lambda_{i_k}H$ corresponding to $\sigma , k=1,\ldots ,p$.
\end{enumerate}
  \end{corollary}
  \begin{proof}
  We will show that under the hypothesis given in the corollary, the conditions given above are equivalent to conditions in Theorem \ref{thm2}. Clearly condition 2 in the above corollary is exactly the same as condition 2 in Theorem \ref{thm2}. First, we are going to show that  $(C,D)$ is controllable if and only if $e_iTD \neq 0$ for all $i=1,2,\ldots , N$.
Since $C$ is diagonalizable, there exists a matrix $T$ such that $TCT^{-1}=J$ where $J=diag \lbrace \lambda_1, \lambda_2, \ldots , \lambda_N \rbrace$.
   Now, 
   \begin{align*}
 TCT^{-1}=J &\Rightarrow TC=JT \\
&\Rightarrow e_iTC=e_iJT\ \forall \ i=1,2,\ldots , N \\
&\Rightarrow (e_iT)C=\lambda_i (e_iT)\ \forall \ i=1,2,\ldots , N   
   \end{align*}
i.e, $e_iT$ is a left eigenvector of $C$ corresponding to $\lambda_i$ for all $i=1,2,\ldots , N$. Then by Lemma \ref{le1}, $(C,D)$ is controllable if and only if $e_iTD \neq 0$ for all $i=1,2,\ldots , N$. In Hao et. al (\cite{hao2018further}) it is proved that for any left eigenpair $(\lambda_i ,t_i)$ of $C$ and $(\mu , \xi )$ of $A+ \lambda_iH$ , $\left( \mu , \xi(t_i \otimes I_n) \right) $ is a left eigenpair of $F$. Therefore in essence, the condition 3 in the above corollary is equivalent to condition 3 in Theorem \ref{thm2}.
  \end{proof}
\begin{remark}
 The existence of the matrix $T$ satisfying all the required condition is crucial in applying the theorem. If the given system is such that $A_i \neq A_j$ for all $i \neq j$, then for $A=blockdiag \lbrace A_1,A_2, \ldots , A_N \rbrace$ to commute with $(T \otimes I)$, $T$ must be a diagonal matrix. If $A_i = A_j$ for some $i \neq j$, then  $T_{ij}$ and $T_{ji}$ are the only possible non-zero elements along with the diagonal entries.
 \end{remark}
\subsection{Controllability in Special Network Topologies}\label{subsection2}
Now we obtain some controllability results over some specific network topologies. If there exists a node $j$ with no incoming edges, then we can derive a necessary condition for controllability of the heterogeneous networked system \eqref{eq:3} as follows.
\begin{theorem}
Suppose that there exists a node $j$ with no edges from any other nodes.
 Then if $(A_j,B)$ is not controllable, then the networked system is not controllable.
\end{theorem}
\begin{proof}
If there exists a node $j$ with no edges from any other nodes, the network topology matrix $C$ is of the form  $$C=\begin{bmatrix}
c_{11} & c_{12} & \ldots  & c_{1N} \\
c_{21} & c_{22} & \ldots & c_{2N} \\
\vdots & \vdots &\vdots & \vdots \\
c_{(j-1)1} & c_{(j-1)2} & \ldots  & c_{(j-1)N} \\
0 & 0 & \ldots & 0 \\
c_{(j+1)1} & c_{(j+1)2} & \ldots  & c_{(j+1)N} \\
\vdots & \vdots &\vdots & \vdots \\
c_{N1} & c_{N2} & \ldots & c_{NN}
\end{bmatrix}$$ 
Suppose that $(A_j,B)$ is not controllable. Then by Lemma 2.1, there exists a non-zero eigenvector $\xi $ of $A_j$ such that $\xi B=0$. The state matrix of the networked system $F$ is given by, $$F=\begin{bmatrix}
A_1+c_{11}H & c_{12}H & \ldots & \ldots & \ldots & c_{1N}H \\
c_{21} & A_2+c_{22}H & \ldots &\ldots & \ldots & c_{2N}H \\
\vdots & \vdots &\vdots & \vdots & \vdots & \vdots \\
0 & 0 & \ldots & A_j & \ldots & 0 \\
\vdots & \vdots &\vdots &\vdots & \vdots & \vdots \\
c_{N1}H & c_{N2}H & \ldots & \ldots & \ldots & A_N+c_{NN}H
\end{bmatrix}$$ and hence $e_j \otimes \xi $ is a left eigenvector of $F$. Since $\xi B=0$, $(e_j \otimes \xi) (D \otimes B)=e_jD \otimes \xi B=0 $. Then the networked system is not controllable.
\end{proof}
We have seen that the controllability of an individual node is necessary when there are no incoming edges to that node. But this is not the case when there are no outgoing edges from a node. Example 5 shows that the controllability of an individual node is not necessary, even if there are no outgoing edges from that node.
\begin{remark}
If there exists some node $j$ with no edges to any other nodes, the controllability of $(A_j,B)$ is not necessary for the controllability of the networked system.
 \end{remark}
 \begin{example}
  Consider a system with $A_1=\begin{bmatrix}
1 & 2 \\
1 & 3
\end{bmatrix}, A_2=\begin{bmatrix}
1 & 1 \\
0 & 0
\end{bmatrix},B=\begin{bmatrix}
1 \\
0
\end{bmatrix}, H=\begin{bmatrix}
0 & 1 \\
1 & 0
\end{bmatrix},C=\begin{bmatrix}
0 & 0 \\
1 & 0
\end{bmatrix}$ and $D=\begin{bmatrix}
1 & 0 \\
0 & 1
\end{bmatrix}$. Clearly there are no edges from node 2 and $(A_2,B)$ is not controllable. But the networked system is controllable.
 \end{example}
 The following theorem gives a situation where the controllability of an individual node with no outgoing edges is necessary.
\begin{theorem}
Suppose that there exists a node $j$ with no edges to any other nodes. If $\xi_i H=0$ for all left eigenvectors of $A_j$, then the controllability of $(A_j,B)$ is necessary for the controllability of the networked system.
\end{theorem}
\begin{proof}
If there exists some node $j$ with no edges to any other nodes, the network topology matrix $C$ is of the form,
 $$C=\begin{bmatrix}
c_{11} & c_{12} & \ldots & c_{1(j-1)} & 0 & c_{1(j+1)} & \ldots & c_{1N} \\
c_{21} & c_{22} & \ldots & c_{2(j-1)} & 0 & c_{2(j+1)} & \ldots & c_{2N} \\
\vdots & \vdots & \ldots & \vdots     &\vdots &\vdots & \ldots & \vdots \\
c_{N1} & c_{N2} & \ldots & c_{N(j-1)} & 0 & c_{N(j+1)} & \ldots & c_{NN}
\end{bmatrix}$$
The state matrix of the networked system $F$ is given by,
 $$F=\begin{bmatrix}
A_1+c_{11}H & c_{12}H & \ldots & 0 & \ldots & c_{1N}H \\
c_{21}H & A_2+c_{22}H & \ldots & 0 & \ldots & c_{2N}H \\
\vdots & \vdots & \ddots & \vdots & \ldots & \vdots \\
c_{j1}H & c_{j2}H & \ldots & A_j & \ldots & c_{jN}H \\
\vdots & \vdots & \ldots & \vdots & \ddots & \vdots \\
c_{N1}H & c_{N2}H & \ldots & 0 & \ldots & A_N+c_{NN}H
\end{bmatrix}$$
Suppose that $(A_j,B)$ is not controllable. Then there exists a non-zero eigenvector $\xi $ of $A_j$ such that $\xi B=0$. Since  $\xi_i H=0$ for all left eigenvectors of $A_j$, $e_j \otimes \xi$ is a left eigenvector of $F$ with $(e_j \otimes \xi) (D \otimes B)=e_jD \otimes \xi B=0 $ and hence the networked system is not controllable.
\end{proof}
\section{Conclusion and future work}\label{section5} 

A necessary and sufficient condition for the controllability of a heterogeneous system under a directed and weighted topology has been derived and substantiated with examples. Controllability results for the networked system over some specific topologies has also been derived. The result is effective and can be verified easily. In the present study the control matrix is uniform in all subsystems. But in the future we intend to study the controllability of the networked systems with heterogeneous control matrices. Another line of research could be an investigation of controllability of networked systems with delays and impulses. However, research in this context is performed, but for homogeneous networked systems with one dimensional communication having delays in control(\cite{muni2019delay}): but for heterogeneous networked systems, such investigation is yet to perform.

\section*{Acknowledgments}
The first author acknowledges with gratitude towards the financial help received from the Council of Scientific and Industrial Research(CSIR), India for the Ph.D work vide letter No.09/1187(0008)/2019-EMR-1 dated 25/07/2019 and the Department of Mathematics, Indian Institute of Space Science and Technology, India, for providing the required support to carry-out this research work.

\end{document}